\documentclass[12pt,reqno]{amsart}
\usepackage[margin=1in]{geometry}
\usepackage[T1]{fontenc}
\usepackage[british]{babel}
\usepackage{lmodern}
\usepackage{amssymb,mathtools,mathrsfs,etoolbox,xspace}
\usepackage[UKenglish,cleanlook]{isodate}
\usepackage{xcolor}
\usepackage[nosort,nocompress]{cite}
\usepackage[breaklinks=true,colorlinks=true,allcolors=black,hypertexnames=false]{hyperref}
\usepackage[nameinlink,capitalize]{cleveref}
\usepackage{orcidlink}
\usepackage{booktabs}
\usepackage[defersetup=false]{microtype}

\mathtoolsset{showonlyrefs}

\makeatletter
\newcommand{\citecomment}[2][]{\citen{#2}#1\citevar}
\newcommand{\citeone}[1]{\citecomment{#1}}
\newcommand{\citetwo}[2][]{\citecomment[, #1]{#2}}
\newcommand{\citevar}{\@ifnextchar\bgroup{; \citeone}{\@ifnextchar[{; \citetwo}{]\xspace}}}
\newcommand{\citefirst}{\@ifnextchar\bgroup{\citeone}{\@ifnextchar[{\citetwo}{]}}}
\newcommand{\cites}{[\citefirst}
\makeatother

\newtheorem{theorem}{Theorem}[section]
\newtheorem{lemma}[theorem]{Lemma}
\newtheorem{proposition}[theorem]{Proposition}
\numberwithin{equation}{section}

\crefname{lemma}{Lemma}{Lemmata}

\renewcommand{\Re}{\mathrm{Re}}

\patchcmd{\section}{\scshape}{\bfseries}{}{}
\makeatletter
\renewcommand{\@secnumfont}{\bfseries}
\newcommand{\tpmod}[1]{{\@displayfalse \pmod{#1}}}
\makeatother

\title{The prime geodesic theorem for the modular orbifold}

\author[Ikuya Kaneko]{Ikuya Kaneko\,\orcidlink{0000-0003-4518-1805}}
\address{Department of Mathematical Sciences, Chalmers University of Technology and the University of Gothenburg, SE-412 96 G\"{o}teborg, Sweden}
\email{ikuya@chalmers.se}



\date{\today}

\ifdefined\pdfinfoomitdate\pdfinfoomitdate=1\fi
\ifdefined\pdftrailerid\pdftrailerid{}\fi

\begin{document}

\begin{abstract}
We establish the prime geodesic theorem for the modular orbifold with exponent $\frac{9}{13} = 0.692307 \cdots$ in the error term. By periodicity and Heath-Brown's quadratic large sieve, we estimate the sixth moment of the error term in an approximation to the Zagier $L$-series at $s = 1$. The smooth explicit formula of Soundararajan and Young then leads to the stated exponent without recourse to subconvexity for quadratic Dirichlet $L$-functions.
\end{abstract}

\maketitle

\section{Introduction}

\subsection{Historical prelude}
The \emph{prime geodesic theorem} describes the asymptotic distribution of oriented primitive closed geodesics on a given hyperbolic orbifold. Let $\Gamma \coloneqq \mathrm{PSL}_{2}(\mathbb{Z})$ be the full modular group, and let $\mathbb{H} \coloneqq \{x+iy \in \mathbb{C}: y > 0\}$ denote the standard upper half-plane. An oriented closed geodesic $P$ on $\Gamma \backslash \mathbb{H}$ corresponds to a hyperbolic conjugacy class in $\Gamma$. Its length is $\log \mathrm{N}(P)$, where $\mathrm{N}(P)$ denotes the norm of the conjugacy class. Writing $P_{0}$ for the primitive geodesic underlying $P$, we define $\Lambda_{\Gamma}(P) \coloneqq \log \mathrm{N}(P_{0})$ and
\begin{equation}
\Psi_{\Gamma}(x) \coloneqq \sum_{\mathrm{N}(P) \leqslant x} \Lambda_{\Gamma}(P).
\end{equation}
The prime geodesic theorem takes~the~form
\begin{equation}
\Psi_{\Gamma}(x) = x+\mathcal{E}_{\Gamma}(x).
\end{equation}
The \emph{explicit formula} of Iwaniec~\cites[Lemma~1]{Iwaniec1984} implies the elementary barrier
\begin{equation}\label{eq:barrier}
\mathcal{E}_{\Gamma}(x) \ll_{\varepsilon} x^{\frac{3}{4}+\varepsilon},
\end{equation}
while the expected bound is
\begin{equation}\label{eq:heuristic}
\mathcal{E}_{\Gamma}(x) \ll_{\varepsilon} x^{\frac{1}{2}+\varepsilon},
\end{equation}
which aligns with Hejhal's $\Omega$-result in~\cites[Chapter~10, Note~18]{Hejhal1983}. The Selberg zeta function takes the place of the Riemann zeta function in this setting, and its analytic properties follow from the Selberg trace formula. The zeros arising from the cuspidal spectrum lie on the~critical line. Their number up to height $T$ is of order $T^{2}$, so attaining~\eqref{eq:heuristic} via the~explicit~formula calls for substantial cancellation in the spectral exponential sum~\cites[Proposition~1.4]{BalkanovaFrolenkovRisager2022}.~For the modular orbifold, counting closed geodesics also has an arithmetic interpretation in terms of averages of real quadratic class numbers ordered by regulator~\cites[Corollary~1.5]{Sarnak1982}.

For the modular surface, Iwaniec~\cites[Theorem~2]{Iwaniec1984} achieved the exponent $\frac{35}{48}$. Under the generalised Lindel\"{o}f hypothesis for quadratic Dirichlet or certain Rankin--Selberg $L$-functions, his argument~\cites[Pages~188--189]{Iwaniec1984-2} yields $\frac{2}{3}$. Luo and Sarnak~\cites[Theorem~1.4]{LuoSarnak1995} improved the exponent to $\frac{7}{10}$. Cai~\cites[Theorem]{Cai2002} synthesised the two approaches to obtain $\frac{71}{102}$. These proofs employ the Kuznetsov formula to estimate spectral exponential sums nontrivially. On the other hand, Soundararajan and Young~\cites[Section~3]{SoundararajanYoung2013} subsequently established
\begin{equation}\phantomsection\label{eq:soundararajan-young}
\mathcal{E}_{\Gamma}(x) \ll_{\varepsilon} x^{\frac{2}{3}+\frac{\vartheta}{6}+\varepsilon},
\end{equation}
where $\vartheta$ is a subconvex exponent for quadratic Dirichlet $L$-functions in the conductor aspect. Their argument combines the Kuznetsov--Bykovski\u{\i} formula with Luo and Sarnak's bound~for spectral exponential sums. Recalling the Weyl exponent $\vartheta = \frac{1}{6}$ of Conrey and Iwaniec~\cites[Corollary~1.5]{ConreyIwaniec2000} yields $\frac{25}{36}$, which remained the best known unconditional exponent for over a decade, until Acosta Reche~\cites[Corollary~1.3]{AcostaReche2026} improved it using zero density estimates for quadratic Dirichlet $L$-functions; see \cref{tab:exponents} for a numerical juxtaposition.

\begin{table}[!htp]
\centering
\setlength\tabcolsep{12pt}
\renewcommand{\arraystretch}{1.2}
\caption{A chronology of unconditional exponents}
\label{tab:exponents}
\begin{tabular}{lccr@{.}l}
	\toprule
	Year & Source & Exponent & \multicolumn{2}{c}{Decimal} \\ \midrule \midrule
	$1984$ & \cites{Iwaniec1984} & $\frac{35}{48}$ & $0$ & $729166 \cdots$ \\ \midrule
	$1995$ & \cites{LuoSarnak1995} & $\frac{7}{10}$ & $0$ & $700000$ \\ \midrule
	$2002$ & \cites{Cai2002} & $\frac{71}{102}$ & $0$ & $696078 \cdots$ \\ \midrule
	$2013$ & \cites{SoundararajanYoung2013} & $\frac{25}{36}$ & $0$ & $694444 \cdots$ \\ \midrule
	$2026$ & \cites{AcostaReche2026} & $\frac{\sqrt{1905}-27}{24}$ & $0$ & $693596 \cdots$ \\ \midrule
	$2026$ & \cref{thm:main} & $\frac{9}{13}$ & $0$ & $692307 \cdots$ \\ \midrule \midrule
\end{tabular}
\end{table}

\subsection{Pointwise improvements}
Unless otherwise specified, the implicit constants in $O(\cdot)$ and $\ll$ may depend on fixed smooth weights and on the fixed constants in hypotheses~involving $\ll$ or $\asymp$; subscripts indicate any additional dependence. An arbitrarily small parameter $\varepsilon > 0$ may change at each occurrence. The main result of the present paper is the following.
\begin{theorem}\label[theorem]{thm:main}
Let $x \geqslant 2$. For any $\varepsilon > 0$,
\begin{equation}\label{eq:main}
\mathcal{E}_{\Gamma}(x) \ll_{\varepsilon} x^{\frac{9}{13}+\varepsilon}.
\end{equation}
\end{theorem}

We consider every fixed even moment of the remainder in the approximation~\eqref{eq:approximation} to the \emph{Zagier $L$-series} $L(1, n^{2}-4)$. The proof combines periodicity with Heath-Brown's \emph{quadratic large sieve} for primitive quadratic characters, then transfers the resulting bounds to the Zagier $L$-series upon expanding the finite Euler factors. The Kuznetsov--Bykovski\u{\i} formula~expresses $\Psi_{\Gamma}(x)$ in terms of weighted sums of these $L$-values. We estimate these sums over short~intervals and use smoothing to obtain the bound for $\mathcal{E}_{\Gamma}(x)$. The sixth moment suffices for~\cref{thm:main}. The same result can also be obtained by a direct spectral argument via the Kuznetsov formula.

For comparison, Acosta Reche~\cites[Section~1.1]{AcostaReche2026} divides the integer traces $n$ of hyperbolic matrices according to whether the Zagier $L$-series $L(s, n^{2}-4)$ has a zero in a rectangle near $\Re(s) = 1$. Zero density estimates suffice to deal with the size of this exceptional set. He uses the Weyl exponent there and a zero-free contour shift for the remaining traces. Our~argument instead exploits the sixth moment over all traces of comparable size to estimate the average absolute value of the Mellin integral on each short interval. Most importantly, our machinery does not rely on subconvexity for quadratic Dirichlet $L$-functions.

For an auxiliary parameter $Y$, combining the sixth moment bound with Soundararajan and Young's \emph{smooth explicit formula}~\cites[Theorem~3.1]{SoundararajanYoung2013} gives an error term roughly of size
\begin{equation}
x^{\frac{7}{8}} Y^{-\frac{1}{4}}+x^{\frac{3}{8}} Y^{\frac{5}{12}}+x^{\frac{1}{12}} Y^{\frac{5}{6}}.
\end{equation}
The first term comes from the classical Luo--Sarnak bound; the other two come from our~short interval bound. The choice $Y = x^{\frac{19}{26}}$ balances the first and third terms at $x^{\frac{9}{13}}$, while the~second term is smaller. Among the fixed even moments, the sixth corresponds to the strongest bound obtainable from this approach.

\subsection{A complementary bound}
A \emph{zero-free region} $\Re(s) > \sigma$, with fixed $\tfrac{1}{2} \leqslant \sigma < 1$, leads to the following bound, parametrised by $\sigma$ to accommodate any further progress towards the quasi-Riemann hypothesis. This result is also briefly mentioned in~\cites[Appendix]{AcostaReche2026}.
\begin{theorem}\leavevmode\phantomsection\label[theorem]{thm:zero-free}
Let $\tfrac{1}{2} \leqslant \sigma < 1$, and suppose that $L(s, \chi) \ne 0$ in $\Re(s) > \sigma$ for every even primitive quadratic character $\chi$. Let $x \geqslant 2$. For any $\varepsilon > 0$,
\begin{equation}\label{eq:zero-free}
\mathcal{E}_{\Gamma}(x) \ll_{\sigma, \varepsilon} x^{\frac{1}{2}+\frac{1}{7-2\sigma}+\varepsilon}.
\end{equation}
\end{theorem}

For $n \geqslant 3$, the fundamental discriminant underlying $n^{2}-4$ is positive, and hence only even characters occur in the hypothesis; see~\eqref{eq:finite-factor}. On every fixed line to the right of $\sigma$, the~zero-free region is capable of bounding the corresponding primitive $L$-functions by an arbitrarily small power of the analytic conductor. Shifting the Mellin contour estimates the error term in the approximation~\eqref{eq:approximation}, and Soundararajan and Young's smoothing argument then implies~\eqref{eq:zero-free}. This improves~\eqref{eq:main} for $\sigma < \tfrac{9}{10}$. In particular, under the generalised Riemann hypothesis~$\sigma = \frac{1}{2}$ for quadratic Dirichlet $L$-functions, the exponent reduces to $\frac{2}{3}$ in agreement with~\eqref{eq:soundararajan-young} under $\vartheta = 0$. Furthermore, assuming the zero-free region with $\sigma = \frac{7}{8}$ announced recently by OpenAI, the stronger exponent $\frac{29}{42} = 0.690476 \cdots$ follows from~\eqref{eq:zero-free} at once.

\subsection*{Acknowledgements}
We thank Gergely Harcos, Dimitrios Chatzakos, and Dmitry Frolenkov for helpful correspondence on earlier versions. Special thanks are owed to Anders S\"{o}dergren for many discussions and his continued support. ChatGPT and Codex assisted with proofreading, editorial revisions, and a survey of the relevant literature. We have independently scrutinised and carefully revised the material and take full responsibility for the final manuscript.

\section{Zagier \texorpdfstring{$L$}{L}-series}\label{sec:averages}
Let $q \in \mathbb{N}$ and let $\delta$ be a discriminant. We define
\begin{equation}\label{eq:rho}
\rho_{q}(\delta) \coloneqq \#\{x \tpmod{2q}: x^{2} \equiv \delta \tpmod{4q} \},
\end{equation}
and
\begin{equation}
\lambda_{q}(\delta) \coloneqq \sum_{q_{1}^{2} q_{2} q_{3} = q} \mu(q_{2}) \rho_{q_{3}}(\delta).
\end{equation}
It follows from M\"{o}bius inversion that
\begin{equation}
\rho_{q}(\delta) = \sum_{q_{1} q_{2} = q} \mu^{2}(q_{2}) \lambda_{q_{1}}(\delta).
\end{equation}
Both are multiplicative in $q$ for fixed $\delta$. The Zagier $L$-series~\cites[Proposition~3]{Zagier1977} is defined~for $\Re(s) > 1$ by
\begin{equation}\label{eq:zagier}
L(s, \delta) \coloneqq \frac{\zeta(2s)}{\zeta(s)} \sum_{q = 1}^{\infty} \frac{\rho_{q}(\delta)}{q^{s}} = \sum_{q = 1}^{\infty} \frac{\lambda_{q}(\delta)}{q^{s}}.
\end{equation}
The series extends meromorphically to the complex plane and vanishes unless $\delta \equiv 0, 1 \tpmod{4}$. For a fundamental discriminant $\delta = d$, the series reduces to the primitive Dirichlet $L$-function $L(s, \chi_{d})$, while at $\delta = 0$ the series equals $\zeta(2s-1)$.

For each integer $n \geqslant 3$, write $n^{2}-4 = d_{n} f_{n}^{2}$, where $d_{n} > 1$ is a fundamental discriminant and $f_{n} \in \mathbb{N}$. We now quote~\cites[Equations~(4)--(5)]{SoundararajanYoung2013} as follows.
\begin{lemma}\leavevmode\phantomsection\label[lemma]{lem:zagier}
Let $\delta = df^{2}$ be a nonzero discriminant, with $d$ fundamental and $f \in \mathbb{N}$. Then
\begin{equation}\label{eq:finite-factor}
L(s, df^{2}) = L(s, \chi_{d}) \sum_{a \mid f} \frac{\mu(a) \chi_{d}(a)}{a^{s}} \sum_{b \mid \frac{f}{a}} \frac{1}{b^{2s-1}},
\end{equation}
where, for any $\varepsilon > 0$, the finite factor is $O_{\varepsilon}(f^{\varepsilon})$ for $\Re(s) \geqslant \tfrac{1}{2}$.
\end{lemma}

To approximate $L(1, n^{2}-4)$, we define, for $U \geqslant 1$,
\begin{equation}\label{eq:mellin-remainder}
\mathcal{I}_{U}(n) \coloneqq \int_{(\frac{1}{2})} L(s, n^{2}-4) U^{s-1} \Gamma(s-1) \, \frac{ds}{2\pi i}.
\end{equation}
For $n \geqslant 3$, Mellin inversion gives~\cites[Equations~(1.7)--(1.8)]{Bykovskii1994}
\begin{equation}\label{eq:approximation}
\sum_{q = 1}^{\infty} \frac{\lambda_{q}(n^{2}-4)}{q} e^{-\frac{q}{U}} = L(1, n^{2}-4)+\mathcal{I}_{U}(n).
\end{equation}
We adapt the following interval form of~\cites[Lemma~2.3]{SoundararajanYoung2013}.
\begin{lemma}\leavevmode\phantomsection\label[lemma]{lem:coefficient-mean}%
Let $I \subset [3, \infty)$ be an interval of length $A$, and let $q = a^{2} b \in \mathbb{N}$ with $b$ squarefree. For any $\varepsilon > 0$,
\begin{equation}
\sum_{n \in I \cap \mathbb{Z}} \lambda_{q}(n^{2}-4) = A \frac{\mu(b)}{b}+O_{\varepsilon}(q^{\frac{1}{2}+\varepsilon}).
\end{equation}
In particular, for $U \geqslant 1$,
\begin{equation}\label{eq:coefficient-mean}
\sum_{n \in I \cap \mathbb{Z}} \sum_{q = 1}^{\infty} \frac{\lambda_{q}(n^{2}-4)}{q} e^{-\frac{q}{U}} = A-\frac{\sqrt{\pi}}{\zeta(\frac{3}{2})} AU^{-\frac{1}{2}}+O_{\varepsilon}(U^{\frac{1}{2}+\varepsilon}+AU^{-1}).
\end{equation}
\end{lemma}

\begin{proof}
Differences of the asymptotic formulae in~\cites[Lemma~2.3]{SoundararajanYoung2013} show the first assertion for half-open intervals. Changing an endpoint costs $O_{\varepsilon}(q^{\frac{1}{2}+\varepsilon})$, by the same formula at consecutive integers. Summing the first statement against $q^{-1} e^{-\frac{q}{U}}$ yields an error term of $O_{\varepsilon}(U^{\frac{1}{2}+\varepsilon})$. By Mellin inversion, the main term is
\begin{equation}
A \sum_{a, b = 1}^{\infty} \frac{\mu(b)}{a^{2} b^{2}} e^{-\frac{a^{2} b}{U}} = A \int_{(1)} U^{s} \Gamma(s) \frac{\zeta(2+2s)}{\zeta(2+s)} \, \frac{ds}{2\pi i}.
\end{equation}
On shifting the contour to $\Re(s) = -1$, the residues at $0$ and $-\frac{1}{2}$ contribute the first two terms in~\eqref{eq:coefficient-mean}, respectively. The pole of $\Gamma(s)$ at $-1$ is cancelled by the zero of $\zeta(2+s)^{-1}$. The remaining integral is $O(AU^{-1})$ by Stirling's formula and $\zeta(1+it)^{-1} \ll \log(2+|t|)$.
\end{proof}

\section{Moments of \texorpdfstring{$\mathcal{I}_{U}(n)$}{the Mellin integral}}\label{sec:moments}
\begin{proposition}\label[proposition]{prop:remainder-moments}
Let $B \geqslant 6$, $U \geqslant 1$, and fix $k \in \mathbb{N}$. Let $I \subset [\tfrac{B}{2}, 2B]$ be an interval of length $A \geqslant 1$. For any $\varepsilon > 0$,
\begin{equation}\label{eq:remainder-moments}
\sum_{n \in I \cap \mathbb{Z}} |\mathcal{I}_{U}(n)| \ll_{k, \varepsilon} B^{\varepsilon} A^{1-\frac{1}{2k}}(1+B^{\frac{1}{k}} U^{-\frac{1}{2}}).
\end{equation}
\end{proposition}

\begin{proof}
For $Q \geqslant 2$, let $\mathscr{X}(Q)$ denote the set of even primitive real characters of conductor $q$ with $2 \leqslant q \leqslant Q$. For $\chi \in \mathscr{X}(Q)$ and $U > 0$, we set
\begin{equation}
\mathcal{P}_{U}(\chi) \coloneqq \sum_{m = 1}^{\infty} \frac{\chi(m)}{m} e^{-\frac{m}{U}}, \qquad \mathcal{R}_{U}(\chi) \coloneqq \mathcal{P}_{U}(\chi)-L(1, \chi).
\end{equation}
We first claim that
\begin{equation}\label{eq:primitive-moments}
\sum_{\chi \in \mathscr{X}(Q)} |\mathcal{R}_{U}(\chi)|^{2k} \ll_{k, \varepsilon} Q^{\varepsilon}(1+QU^{-k}).
\end{equation}
By periodicity, the partial sums of $\chi$ are $O(q)$, and partial summation shows $\mathcal{R}_{U}(\chi) \ll \frac{q}{U}$ for any $U > 0$. Let $Z = Q^{2}$. For $U \geqslant Z$, bounding trivially yields
\begin{equation}
\sum_{\chi \in \mathscr{X}(Q)} |\mathcal{R}_{U}(\chi)|^{2k} \ll Q^{2k+1} U^{-2k} \leqslant 1.
\end{equation}
For $0 < U < Z$, it is convenient to employ the exact identity $\mathcal{R}_{U} = \mathcal{R}_{Z}+\mathcal{P}_{U}-\mathcal{P}_{Z}$. We split $\mathcal{P}_{U}-\mathcal{P}_{Z}$ into dyadic subintervals $M \leqslant m < 2M$, where $M = 2^{j}$ and $j \geqslant 0$. The coefficient of $\chi(m)$ on each interval satisfies
\begin{equation}
\frac{e^{-\frac{m}{U}}-e^{-\frac{m}{Z}}}{m} \ll \frac{e^{-\frac{M}{Z}} \min(1, MU^{-1})}{M}.
\end{equation}
The $k$-th power has length $O_{k}(M^{k})$ and coefficients $\ll \tau_{k}(m) M^{-k} \min(1, MU^{-1})^{k} e^{-\frac{kM}{Z}}$. We apply the quadratic large sieve with no squarefree restriction~\cites[Corollary~3]{HeathBrown1995} to this~power together with the divisor bound. The $2k$-th moment of the original polynomial is therefore
\begin{equation}
\ll_{k, \varepsilon} (QM)^{\varepsilon}(1+QM^{-k}) \min(1, MU^{-1})^{2k} e^{-\frac{2kM}{Z}} \ll_{k, \varepsilon} (QM)^{\varepsilon}(1+QU^{-k}) e^{-\frac{2kM}{Z}}.
\end{equation}
Since $\sum_{j = 0}^{\infty} 2^{j\varepsilon} e^{-\frac{2^{j}}{Z}} \ll_{\varepsilon} Z^{\varepsilon}$ and $Z = Q^{2}$, Minkowski's inequality in $\ell^{2k}(\mathscr{X}(Q))$ combines all dyadic subintervals at the cost of an arbitrarily small power of $Q$. Adding the contribution of $\mathcal{R}_{Z}$ now confirms~\eqref{eq:primitive-moments}.

We now insert the finite factor in~\eqref{eq:finite-factor} into~\eqref{eq:mellin-remainder}. Mellin inversion gives
\begin{equation}\label{eq:remainder-expansion}
\mathcal{I}_{U}(n) = \sum_{ab \mid f_{n}} \frac{\mu(a) \chi_{d_{n}}(a)}{ab} \mathcal{R}_{\frac{U}{ab^{2}}}(\chi_{d_{n}}).
\end{equation}
Since $f_{n} \leqslant 2B$, the number of divisor pairs is $\tau_{3}(f_{n}) \ll_{\varepsilon} B^{\varepsilon}$. For each fixed pair, we observe $d_{n} \leqslant (\frac{2B}{ab})^{2}$. For a fixed value $d$ of $d_{n}$, let $\eta > 1$ denote the smallest unit of norm $1$ in~$\mathbb{Q}(\sqrt{d})$, and write $n_{j} \coloneqq \eta^{j}+\eta^{-j}$ for $j \in \mathbb{Z}$. Pell's parametrisation as in the proof of~\cites[Lemma~4.2]{Sarnak1982} demonstrates that every trace under consideration equals $n_{j}$ for some $j \geqslant 1$. Since $\eta+\eta^{-1}$ is an integer at least $3$ and $n_{j-2} \leqslant n_{j}$ for $j \geqslant 1$, it follows that
\begin{equation}
n_{j+2} = (\eta^{2}+\eta^{-2}) n_{j}-n_{j-2} \geqslant 6n_{j}.
\end{equation}
Hence, for each fixed $d$, at most two traces lie in $[\tfrac{B}{2}, 2B]$, whose endpoints have ratio $4$.

For each pair that occurs, we recall~\eqref{eq:primitive-moments} with $Q = (\frac{2B}{ab})^{2}$ and smoothing length $\frac{U}{ab^{2}}$. H\"{o}lder's inequality over the divisor pairs then gives
\begin{equation}
\sum_{\frac{B}{2} \leqslant n \leqslant 2B} |\mathcal{I}_{U}(n)|^{2k} \ll_{k, \varepsilon} B^{\varepsilon} \sum_{a = 1}^{\infty} \sum_{b = 1}^{\infty} \{(ab)^{-2k}+B^{2} U^{-k} a^{-k-2} b^{-2} \},
\end{equation}
where both sums converge. Since $\#(I \cap \mathbb{Z}) \leqslant 2A$, H\"{o}lder's inequality completes the proof.
\end{proof}

\section{Proof of \texorpdfstring{\cref{thm:main}}{the pointwise theorem}}\label{sec:main-proof}
\begin{lemma}\leavevmode\phantomsection\label[lemma]{lem:short-interval}
Let $x \geqslant 2$ and $\sqrt{x} \leqslant h \leqslant \tfrac{x}{2}$, and fix $k \in \mathbb{N}$. For any $\varepsilon > 0$,
\begin{equation}\label{eq:short-interval}
\Psi_{\Gamma}(x+h)-\Psi_{\Gamma}(x) = h+O_{k, \varepsilon}(x^{\frac{1}{4}+\frac{3}{8k}+\varepsilon} h^{\frac{1}{2}-\frac{1}{4k}}+x^{\frac{1}{4k}+\varepsilon} h^{1-\frac{1}{2k}}).
\end{equation}
\end{lemma}

\begin{proof}
It suffices to consider sufficiently large $x$ in what follows. Let $1 \leqslant U \leqslant x$, $X \coloneqq \sqrt{x}+\frac{1}{\sqrt{x}}$ and $X^{\prime} \coloneqq \sqrt{x+h}+\frac{1}{\sqrt{x+h}}$. The Kuznetsov--Bykovski\u{\i} formula~\cites[Proposition~2.2]{SoundararajanYoung2013} implies
\begin{equation}
\Psi_{\Gamma}(x+h)-\Psi_{\Gamma}(x) = 2 \sum_{X < n \leqslant X^{\prime}} \sqrt{n^{2}-4} L(1, n^{2}-4).
\end{equation}
The interval $(X, X+\tfrac{h}{\sqrt{x}}]$ contains $(X, X^{\prime}]$, lies in $[\tfrac{X}{2}, 2X]$, and has length $\frac{h}{\sqrt{x}}$. Applying~\eqref{eq:remainder-moments} to this interval shows uniformly for $X \leqslant y \leqslant X^{\prime}$ that
\begin{equation}\label{eq:trace-remainder}
\sum_{X < n \leqslant y} |\mathcal{I}_{U}(n)| \ll_{k, \varepsilon} x^{-\frac{1}{2}+\frac{1}{4k}+\varepsilon} h^{1-\frac{1}{2k}}+x^{-\frac{1}{2}+\frac{3}{4k}+\varepsilon} h^{1-\frac{1}{2k}} U^{-\frac{1}{2}}.
\end{equation}
We apply \cref{lem:coefficient-mean} on $(X, y]$, recall~\eqref{eq:approximation}, and note that the coefficient of the interval length in~\eqref{eq:coefficient-mean} is $1+O(U^{-\frac{1}{2}})$. Partial summation with the increasing weight $\sqrt{n^{2}-4} \ll x^{\frac{1}{2}}$ gives
\begin{equation}
\Psi_{\Gamma}(x+h)-\Psi_{\Gamma}(x) = 2 \int_{X}^{X^{\prime}} \sqrt{v^{2}-4} \, dv+O_{k, \varepsilon}(x^{\frac{1}{2}+\varepsilon} U^{\frac{1}{2}}+x^{\frac{1}{4k}+\varepsilon} h^{1-\frac{1}{2k}}+x^{\frac{3}{4k}+\varepsilon} h^{1-\frac{1}{2k}} U^{-\frac{1}{2}}).
\end{equation}
Making a change of variables $v \mapsto \sqrt{t}+\frac{1}{\sqrt{t}}$ yields
\begin{equation}
2 \int_{X}^{X^{\prime}} \sqrt{v^{2}-4} \, dv = \int_{x}^{x+h} \Big(1-\frac{1}{t} \Big)^{2} \, dt = h+O \Big(\frac{h}{x} \Big).
\end{equation}
Taking $U = x^{\frac{3}{4k}-\frac{1}{2}} h^{1-\frac{1}{2k}} \in [1, x]$ establishes the desired claim.
\end{proof}

\begin{proof}[Proof of \cref{thm:main}]
It suffices to consider sufficiently large values of $x$. Fix a nonnegative function $w \in C_{c}^{\infty}((\tfrac{1}{2}, 1))$ with integral $1$, and take $Y = x^{\frac{19}{26}}$ and $k(u) \coloneqq Y^{-1} w(\frac{u}{Y})$. For a fixed sufficiently small $\eta > 0$ and sufficiently large $x$, we have $x^{\frac{1}{2}+\eta} \leqslant Y \leqslant \tfrac{x}{\log x}$. On the support of $k$, we have $\tfrac{Y}{2} < u < Y$, so $\sqrt{x} \leqslant u \leqslant \tfrac{x}{2}$, as required by \cref{lem:short-interval} with $h = u$. It follows from~\cites[Theorem~3.1]{SoundararajanYoung2013} that
\begin{equation}\label{eq:smoothing}
\int_{0}^{Y} \Psi_{\Gamma}(x+u) k(u) \, du = x+\int_{0}^{Y} u k(u) \, du+O_{\varepsilon}(x^{\frac{7}{8}+\varepsilon} Y^{-\frac{1}{4}}),
\end{equation}
where the error term is $O_{\varepsilon}(x^{\frac{9}{13}+\varepsilon})$. Meanwhile, the case $k = 3$ of \cref{lem:short-interval} shows uniformly throughout the support of $k$ that
\begin{equation}
\Psi_{\Gamma}(x+u)-\Psi_{\Gamma}(x) = u+O_{\varepsilon}(x^{\frac{3}{8}+\varepsilon} Y^{\frac{5}{12}}+x^{\frac{1}{12}+\varepsilon} Y^{\frac{5}{6}}) = u+O_{\varepsilon}(x^{\frac{9}{13}+\varepsilon}).
\end{equation}
Substitution into the identity~\cites[Equation~(14)]{SoundararajanYoung2013} establishes the desired claim.
\end{proof}

\section{Zero-free regions}\label{sec:zero-free}
One may assume throughout this section that $x$ is sufficiently large. Fix $0 < \varepsilon \leqslant 1$. Let $\delta \coloneqq 1-\sigma$, $B \geqslant 6$, $n \in \mathbb{Z} \cap [\tfrac{B}{2}, 2B]$, and $1 \leqslant U \leqslant B^{2}$. It is convenient to set $\alpha \coloneqq \min(\frac{\delta}{4}, \frac{\varepsilon}{16})$. Using~\cites[Lemma~1]{Bykovskii1994}\footnote{In~\cites[Lemma~1]{Bykovskii1994} and its proof, replace $\sigma_{0}+\sigma$ by $\sigma_{0}+\delta$ and $\frac{\sigma}{2}$ by $\frac{\delta}{2}$ in $r_{3}$, respectively.} for the zero-free rectangle with left edge $\sigma+\frac{\alpha}{2}$ and height $2+|t|$~gives
\begin{equation}\label{eq:zero-free-line}
L(\sigma+\alpha+it, \chi_{d_{n}}) \ll_{\sigma, \varepsilon} (d_{n}(1+|t|))^{\frac{\varepsilon}{8}}.
\end{equation}
Recall that the finite factor in~\eqref{eq:finite-factor} is $O_{\varepsilon}(f_{n}^{\frac{\varepsilon}{8}})$. Since $d_{n} f_{n}^{2} \leqslant 4B^{2}$ and $U \leqslant B^{2}$, shifting the contour in~\eqref{eq:mellin-remainder} to $\Re(s) = \sigma+\alpha < 1$ crosses no pole and gives, by Stirling's formula,
\begin{equation}\label{eq:zero-free-mellin}
\mathcal{I}_{U}(n) \ll_{\sigma, \varepsilon} B^{\varepsilon} U^{\sigma-1}.
\end{equation}
This also covers $\sigma = \frac{1}{2}$, since the contour lies strictly to the right of the critical line.

Reuse the function $w$ from the proof of \cref{thm:main} and take
\begin{equation}
U = x^{\frac{2}{7-2\sigma}}, \qquad Y = x^{\frac{13-6\sigma}{2(7-2\sigma)}}.
\end{equation}
Set $B = \sqrt{x}+\frac{1}{\sqrt{x}}$ and $k(u) \coloneqq Y^{-1} w(\frac{u}{Y})$. Partial summation as in the proof of \cref{lem:short-interval}, with~\eqref{eq:zero-free-mellin} in place of~\eqref{eq:trace-remainder}, shows uniformly throughout the support of $k$ that
\begin{equation}
\Psi_{\Gamma}(x+u)-\Psi_{\Gamma}(x) = u+O_{\sigma, \varepsilon}(x^{\frac{1}{2}+\varepsilon} U^{\frac{1}{2}}+x^{\varepsilon} YU^{\sigma-1}).
\end{equation}
The second term on the right-hand side of~\eqref{eq:coefficient-mean} contributes $O(YU^{-\frac{1}{2}})$, which is absorbed by $YU^{\sigma-1}$ because $\sigma \geqslant \tfrac{1}{2}$. Both terms as well as the error term in~\eqref{eq:smoothing} have the desired~exponent $\frac{1}{2}+\frac{1}{7-2\sigma}$. Note that for any fixed $\tfrac{1}{2} \leqslant \sigma < 1$, our choices of parameters satisfy the conditions $1 \leqslant U \leqslant x$ and $x^{\frac{1}{2}+\eta} \leqslant Y \leqslant \tfrac{x}{\log x}$ for a fixed sufficiently small $\eta > 0$. \cref{thm:zero-free} now~follows.

\begingroup
\interlinepenalty=10000
\providecommand{\MR}{\relax\ifhmode\unskip\space\fi MR }
\providecommand{\Zbl}{\relax\ifhmode\unskip\space\thinspace\fi Zbl }
\providecommand{\arXiv}{\relax\ifhmode\unskip\space\thinspace\fi arXiv }

\endgroup
\enlargethispage{4\baselineskip}

\end{document}